\documentclass[a4paper,11pt]{amsart}
\usepackage{amssymb}
\usepackage{cmap}
\usepackage[pdfdisplaydoctitle=true,
            colorlinks=true,
            urlcolor=blue,
            citecolor=blue,
            linkcolor=blue,
            pdfstartview=FitH,
            pdfpagemode= UseNone,
            bookmarksnumbered=true]{hyperref}
\usepackage{todonotes}
\usepackage{mathtools}
\usepackage{tikz-cd}
\mathtoolsset{showonlyrefs}
\usepackage{enumerate}
\usepackage{soul,xcolor}
\setstcolor{red}

\usepackage{lineno}

\theoremstyle{plain}
\newtheorem{thm}{Theorem}[section]
\newtheorem{cor}[thm]{Corollary}
\newtheorem{lem}[thm]{Lemma}
\newtheorem{prop}[thm]{Proposition}

\theoremstyle{definition}

\theoremstyle{definition}
\newtheorem{rem}[thm]{Remark}

\newtheorem*{ass*}{Assumption}

\numberwithin{equation}{section}

\newcommand{\NN}{\mathbb{N}} 
\newcommand{\RR}{\mathbb{R}} 

\newcommand{\Met}{\mathrm{Met}}			

\newcommand{\Rn}{\mathbb{R}^{n}}                    

\DeclareMathOperator{\tr}{tr} %

\let\on=\operatorname

\def\SO{\on{SO}}

\def\distO{\on{dist}_{\Omega_+^1}}
\def\distnm{\on{dist}^{\Omega^1_+}_x}
\def\distnm{\on{dist}_{n\times m}}
\def\Sym{\on{Sym}}
\def\distMet{\on{dist}_{\Met}}
\def\distSym{\on{dist}_{m\times m}}

\def\lin{\on{lin}}

\allowdisplaybreaks

\hypersetup{
    pdftitle={The Metric Completion of the Space of Vector-Valued One-Forms}, 
    pdfauthor={Zhe}, 
    pdfsubject={MSC 2010: 58D05, 35Q35}, 
    pdfkeywords={Sobolev metrics, surfaces, shape analysis}, 
    pdflang=EN 
    }
\begin{document}

\title[The space of full-ranked one-forms]{The Metric Completion of the Space of Vector-Valued One-Forms}

\author{Nicola Cavallucci}
\address{Karlsruhe Institute of Technology, Engelstrasse 2, D-76128, Karlsruhe, Germany}
\email{n.cavallucci23@gmail.com}

\author{Zhe Su}
\address{Department of Neurology, University of California, Los Angeles, Los Angeles, CA, USA}
\email{zsu20@g.ucla.edu}


\date{\today}


\begin{abstract}
	The space of full-ranked one-forms on a smooth, orientable, compact manifold (possibly with boundary) is metrically incomplete with respect to the induced geodesic distance of the generalized Ebin metric. We show a distance equality between the induced geodesic distances of the generalized Ebin metric on the space of full-ranked one-forms and the corresponding Riemannian metric defined on each fiber. Using this result we immediately have a concrete description of the metric completion of the space of full-ranked one-forms. Additionally, we study the relationship between the space of full-ranked one-forms and the space of all Riemannian metrics, leading to quotient structures for the space of Riemannian metrics and its completion.
\end{abstract}


\maketitle

\section{Introduction}
Let $M$ be a smooth, orientable, compact manifold (possibly with boundary) of dimension $m$. The space of smooth $\Rn$-valued full-ranked one-forms $\Omega^1_+(M,\Rn)$ with $n\geq m$ on $M$ is an infinite dimensional Fr{\'e}chet manifold with tangent space the space of all smooth $\Rn$-valued one-forms $\Omega^1(M,\Rn)$. In \cite{bauer2021oneforms} Bauer et al. defined a diffeomorphism-invariant $L^2$ Riemannian metric, called the generalized Ebin metric, on the space of full-ranked one-forms $\Omega^1_+(M,\Rn)$. The metric results in interesting geometry: it induces a non-degenerate geodesic distance, which is a non-trivial result for infinite-dimensional manifolds as the induced geodesic distance of a Riemannian metric may fail to be positive definite \cite{eliashberg1993bi,michor2005vanishing,bauer2013geodesic,bauer2018vanishing}; there are explicit solutions for the geodesic initial value problem with respect to the $L^2$ metric; the sectional curvature, depending on $m$ and $n$, is either sign-semidefinite or admit both signs; the space of full-ranked one-forms is geodesically and metrically incomplete. Moreover, the $L^2$ metric on the space of full-ranked one-forms is linked via a Riemannian submersion to the Ebin metric \cite{ebin1970manifold, DeWitt67} on the space of all Riemannian metrics, which has a wide range of applications in general relativity, Teichm{\"u}ller theory and mathematical statistics and has been well studied for its basic geometry in \cite{gil1991riemannian, freed1989basic, clarke2010metric, clarke2011riemannian, clarke2013completion, clarke2013geodesics, cavallucci2022}. The generalized Ebin metric, due to its diffeomorphism-invariant property, also has been proposed for shape analysis of curves and surfaces \cite{bauer2021oneforms,su2020shape}. 

In this paper, we aim to delve deeper into the induced metric structure of the generalized Ebin metric and give a concrete description of the metric completion of the space of full-ranked one-forms, along with its relationship to the metric completion of the space of Riemannian metrics. For the whole paper we restrict ourselves to the case $n>m$, since when $n=m$ the fiber space at every point $x\in M$ of the space of full-ranked one-forms has two connected components, requiring the metric to be studied separately on each component, which increases the complexity.

For the space of Riemannian metrics, Clarke has proved a distance equality in \cite{clarke2013geodesics} for the induced geodesic distance of the Ebin metric, and characterized its metric completion in \cite{clarke2013completion} using a notion of $\omega$-convergence. Although the same strategies can be applied to our case, here we provide simpler and shorter proofs for a distance equality \eqref{eq.dist-equality}, in a way similar to what the first author did in \cite{cavallucci2022} for the space of Riemannian metrics. This distance equality forms the main result of this paper and leads directly to an explicit characterization of the metric completion of the space of full-ranked one-forms.

The generalized Ebin metric on the space of full-ranked one-forms is given by the integral over $M$ of Riemannian metrics defined on each fiber, i.e., the space of full rank tensors $(\Rn\otimes T_x^*M)_+$, see Section~\ref{sec.L2metric} for a definition of the metric. Denote by $\distO$ the induced geodesic distance function of the generalized Ebin metric \eqref{eq.metric.OneForms} on $\Omega^1_+(M,\Rn)$ and by $d_x$ the induced geodesic distance function of the corresponding Riemannian metric \eqref{eq-metric} on $(\Rn\otimes T_x^*M)_+$, depending on a fixed Riemannian metric $g_0$ on $M$. Let $d\mu_{g_0}$ be the induced volume form of $g_0$. The main result of the paper is as follows. 
\begin{thm}\label{theo-intro-dist-equality}
	Let $\alpha,\beta \in \Omega^1_+(M,\mathbb{R}^n)$. Then
	\begin{equation}\label{eq.dist-equality}
		\distO(\alpha,\beta)^2 = \int_Md_x(\alpha(x),\beta(x))^2d\mu_{g_0}(x).
	\end{equation}
\end{thm}
It is worth noting that the distance $\distO$ of the generalized Ebin metric \eqref{eq.metric.OneForms} is independent of the choice of $g_0$, so it is the right hand side of \eqref{eq.dist-equality}.
The metric space of full rank tensors $((\Rn\otimes T_x^*M)_+, d_x)$ is isometric to the space of $n\times m$ full rank matrices $(M_+(n,m), \distnm)$ under a standard local trivialization, where $\distnm$ is the induced geodesic distance of the Riemannian metric \eqref{eq-metric-standard} on the space of full rank matrices $M_+(n,m)$. A definition of a standard local trivialization can be found in Section~\ref{sec.L2metric}. The metric completion for the finite dimensional space of full rank matrices $M_+(n,m)$ with respect to the geodesic distance $\distnm$, denoted by $\overline{M_+(n,m)}$, has been characterized in \cite{bauer2021oneforms}: it is given by the space of all $n\times m$ matrices modulo the equivalence relation that two matrices are identified if they both are not full rank. However, for the infinite dimensional space of full-ranked one-forms $\Omega^1_+(M,\Rn)$, only a conjecture was given in \cite{bauer2021oneforms} for its metric completion with respect to the geodesic distance $\distO$. Using Theorem \ref{theo-intro-dist-equality} we immediately have a concrete characterization of the metric completion of the space of full-ranked one-forms $\Omega^1_+(M,\Rn)$.
\begin{thm}
	Let $M$ be a smooth, orientable, compact manifold (possibly with boundary) of dimension $m$ and let $n>m$. The metric completion of the space of full ranked one-forms $\Omega^1_+(M,\Rn)$ is isometric to $L^2(M,\overline{M_+(n,m)})$. Therefore it depends only on $n$ and $m$.
\end{thm}

We demonstrate later that the space of Riemannian metrics $\Met(M)$, equipped with the geodesic distance induced by a multiple of the Ebin metric, is isometric to a quotient of the space of full-ranked one-forms $\Omega^1_+(M,\Rn)$ with the geodesic distance induced by the generalized Ebin metric. This result provides a global analogy of the relationship between the space of positive definite symmetric matrices and the space of full rank matrices \cite{bauer2021oneforms}. Using the descriptions of the metric completions of the space of full-ranked forms from this paper and the space of Riemannian metrics from \cite{cavallucci2022}, we observe a similar phenomenon in their respective metric completions. 

The paper is organized as follows. In Section~\ref{sec.L2metric} we first recall the definition of the generalized Ebin metric on the space of full-ranked one-forms $\Omega^1_+(M,\Rn)$ and then complete the proofs of Theorem~\ref{theo-intro-dist-equality}. Using this theorem, we then give a concrete description of the metric completion of the space of full-ranked one-forms $\Omega^1_+(M,\Rn)$, and confirm the conjecture in \cite{bauer2021oneforms}. In Section~\ref{sec.EbinMetric}, we explore the relationship between the space of Riemannian metrics and the space of full-ranked one-forms. We define a multiple of the Ebin metric on the space of Riemannian metrics and then show a quotient structure of the space of Riemannian metrics. Finally, we extend this result to their metric completions using their $L^2$ mapping space formulations and give a quotient characterization of the metric completion of the space of Riemannian metrics.

\subsection*{Acknowledgments}
N. Cavallucci was partially supported by the SFB/TRR 191, funded by the DFG. Z. Su was supported by NIH/NIAAA award R01-AA026834.


\section{The metric on the space of full-ranked one-forms}\label{sec.L2metric}
Let $\Omega^1(M,\Rn)$ denote the space of smooth $\Rn$-valued one-forms on $M$, i.e., the Fr\'echet vector space of smooth sections of the tensor product bundle $\Rn\otimes T^*M$, where $M$ is a smooth, orientable, compact manifold (possibly with boundary) of dimension $m$ with $m\leq n$. In local coordinates $\{x^i, i=1,2,\cdots, m\}$, with respect to the standard basis $\{e_i, 1\leq i\leq n\}$ of $\Rn$ and the basis $\{dx^i, 1\leq i\leq m\}$ of $T^*_xM$, each one-form $\alpha\in\Omega^1(M,\Rn)$ at $x\in M$ can be represented by a $n\times m$ matrix, namely $\alpha = \alpha^i_j e_i\otimes dx^j$. So, called $M(n,m)$ the space of all $n\times m$ matrices, a local trivialization of $\Rn\otimes T^*M$ is given by a commutative diagram
\begin{equation}
	\label{diagram}
	\begin{tikzcd}
		\pi^{-1}(U) \arrow[r, "\varphi"] \arrow[d, "\pi"'] 
		& U \times M(n,m) \arrow[ld, "\pi_1"] \\ 
		U 
	\end{tikzcd}
\end{equation}
where $\pi_1$ is the projection on the first coordinate, $U$ is an open subset of $M$ and $\varphi$ is a diffeomorphism of vector bundles, in particular the restriction of $\varphi$ at each fiber is a linear map.

The space of full-ranked $\mathbb{R}^n$-valued one-forms is the subset of $\Omega^1(M,\mathbb{R}^n)$ defined by
\begin{align}
	\Omega^1_+(M,\mathbb{R}^n) := \lbrace \alpha \in \Omega^1(M,\mathbb{R}^n) \text{ s.t. } \alpha(x) \text{ has rank } m \text{ for all } x \in M \rbrace,
\end{align}
which is an open subset of the Fr\'echet vector space $\Omega^1(M,\Rn)$, thus it is a smooth Fr\'echet manifold \cite{hamilton1982inverse} and the tangent space at each point is canonically identified with $\Omega^1(M,\Rn)$.
The condition of having rank $m$ can be checked using any local trivialization of $\Rn\otimes T^*M$. Indeed if $\alpha\in \Omega^1(M,\Rn)$ and $\varphi$ is a local trivialization around $x\in M$ then $\alpha(x)$ has rank $m$ if and only if $(\pi_2\circ \varphi \circ \alpha) (x) \in M(n,m)$ has full rank. Here $\pi_2$ is the projection on the second factor in diagram \eqref{diagram}. The space $\Omega^1_+(M,\mathbb{R}^n)$ is the space of smooth sections of the fiber bundle $(\Rn\otimes T^*M)_+$ over $M$, whose fibers are the set $M_+(n,m)$ of full rank $n\times m$ matrices.
A corresponding local trivialization for $(\Rn\otimes T^*M)_+$ is a commutative diagram
\begin{center}
	\begin{tikzcd} 
		\pi^{-1}(U) \arrow[r, "\varphi"] \arrow[d, "\pi"'] 
		& U \times M_+(n,m) \arrow[ld, "\pi_1"] \\ 
		U 
	\end{tikzcd}
\end{center}
By the discussion above we see that any local trivialization of $(\Rn\otimes T^*M)$ induces a local trivialization of $(\Rn\otimes T^*M)_+$.

\subsection{The \texorpdfstring{$L^2$}{Lg} metric on the space of full-ranked one-forms}
We recall the definition of the $L^2$ metric introduced in \cite{bauer2021oneforms} on $\Omega^1_+(M,\Rn)$. If $\alpha\in\Omega^1_+(M,\Rn)$ then $\alpha^T = \alpha^i_jdx^i\otimes e_j$ in local coordinates. The $(2,1)$ contraction $\alpha^T\alpha := C(\alpha^T\otimes\alpha)$ of the tensor product $\alpha^T\otimes\alpha$ with respect to the Euclidean scalar product on $\Rn$ defines a Riemannian metric on $M$. For each $x\in M$, the metric $\alpha^T\alpha$ induces an inner product $(\alpha^T\alpha)^{-1}(x)$ on the cotangent space $T_x^*M$. The Euclidean scalar product on $\Rn$, together with the inner product $(\alpha^T\alpha)^{-1}(x)$ on $T_x^*M$, thus yield a scalar product on the tensor product space $\Rn\otimes T_x^*M$:
\begin{align}
	\tr\left(\zeta(x)(\alpha^T\alpha)^{-1}(x)\eta^T(x)\right)
\end{align}
for $\zeta,\eta\in T_{\alpha}\Omega^1_+(M,\Rn)$. The $L^2$ metric on the space of full-ranked one-forms $\Omega^1_+(M,\Rn)$ is then defined as the integral over $M$ of this scalar product with respect to the volume form induced by the metric $g = \alpha^T\alpha$:
\begin{align}\label{eq.metric.OneForms}
	G_{\alpha}(\zeta, \eta) &= \int_M\tr\left(\zeta(x)(\alpha^T\alpha)^{-1}(x)\eta^T(x)\right)d\mu_{g}(x). 
\end{align}
Here we omit $x$ in the integrand for simplicity and $d\mu_{g} = \sqrt{\det(g)}dx$ in local coordinates. 
\begin{rem}\label{rem.penroseinv}
	Using the pointwise Moore-Penrose inverse of $\alpha$, we can obtain an alternative formula for the integrand of the metric \eqref{eq.metric.OneForms}. The Moore-Penrose inverse of $\alpha$, for each $x\in M$, is given by the tensor obtained from $\alpha^T$ by raising the lower index with the Euclidean metric $\delta$ and lowering the upper index with the induced metric $g$ of $\alpha$, and locally $\alpha^+ = g^{is}\alpha_s^k\delta_{kj} \frac{\partial}{\partial x^i}\otimes e^j $. Here $\{e^i \}$ denotes the dual basis for $\{e_i\}$ on $\RR^n$. (Alternatively, the Moore-Penrose inverse $\alpha^+$ can be seen as the unique adjoint operator of $\alpha$ satisfying $\delta(\alpha(h), w) = g(h, \alpha^+(w))$ with $h\in T_xM$, $w\in\Rn$, and it has property that $\alpha^+\alpha$ gives the identity endomorphism of $T_xM$.) An alternative form of the integrand of the metric \eqref{eq.metric.OneForms} is then given by:
	\begin{align}
		\tr\left(\zeta(\alpha^T\alpha)^{-1}\eta^T\right) = \tr\left(\boldsymbol{\zeta}\boldsymbol{\eta}^T\right),
	\end{align}
	where $\boldsymbol{\zeta} = \zeta\alpha^+$ and $\boldsymbol{\eta} = \eta\alpha^+$ simply denote the contractions of tensors or the compositions of mappings.
\end{rem}
The formula \eqref{eq.metric.OneForms} clearly shows that the metric $G$ is independent of the original metric on $M$. Another expression for the metric $G$ can be given with the help of a fixed Riemannian metric $g_0$ on $M$, whose volume form is denoted by $\mu_{g_0}$. On each fiber $( \Rn\otimes T^*_xM)_+$ we consider the following Riemannian metric. For $a\in ( \Rn\otimes T^*_xM)_+$ and $u,v \in T_a( \Rn\otimes T^*_xM)_+ \cong  \Rn\otimes T^*_xM$ we set
\begin{equation}
	\label{eq-metric}
	\langle u,v \rangle_{a,x} := \text{tr}(u(a^Ta)^{-1}v^T) \sqrt{\det (g_0(x)^{-1} (a^Ta))}.
\end{equation}
Then
\begin{equation}
	\label{eq-metric-g0}
	G_{\alpha}(\zeta, \eta) = \int_M\langle \zeta(x), \eta(x)\rangle_{\alpha(x),x}d\mu_{g_0}(x).
\end{equation}

A local trivialization $\varphi$ of the bundle $\Omega^1(M,\RR^n)$ induced by a local chart on an open set $U$ sending the volume form $\mu_{g_0}$ to the Euclidean one is called \emph{standard}. Let $\varphi\vert_x \colon ( \Rn\otimes T^*_xM) \to M(n,m)$ be the restriction of $\varphi$ to the fiber at $x\in U$. Then the pushforward of the metric on $( \Rn\otimes T^*_xM)$, namely $(\varphi\vert_x)_*(\langle \cdot,\cdot \rangle_{\cdot,x})$, defines the following Riemannian metric on $M_+(n,m)$:
\begin{equation}
	\label{eq-metric-standard}
	\langle U,V \rangle_{A} = \text{tr}(U(A^TA)^{-1}V^T) \sqrt{\det (A^TA)},
\end{equation}
where $A\in M_+(n,m)$ and $U,V\in T_AM_+(n,m) \cong M(n,m)$. We denote this metric by $g_{n\times m}$.

The Riemannian structures $\langle \cdot, \cdot \rangle_{\cdot, x}$ on $( \Rn\otimes T^*_xM)_+$, $g_{n\times m}$ on $M_+(n,m)$ and $G$ on $\Omega^1_+(M,\Rn)$ induce corresponding Riemannian distances $d_x$ on $( \Rn\otimes T^*_xM)$, $\distnm$ on $M_+(n,m)$ and $\distO$ on $\Omega^1_+(M,\Rn)$, by taking the infimums of the lengths of piecewise $C^1$-paths connecting two elements. From the discussion above it follows that if $\varphi$ is a standard local trivialization around $x\in M$ then $\varphi\vert_x$ is an isometry between $d_x$ and $\distnm$.

Our main result, Theorem \ref{theo-intro-dist-equality}, says that \eqref{eq-metric-g0} implies that $\distO$ coincides with the integral of the distances $d_x$. Before proving it we present some easy results on the metric space $(M_+(n,m), \distnm)$.

\subsection{The space \texorpdfstring{$M_+(n,m)$}{Lg}} 
A matrix $A\in M(n,m)$ has rank smaller than $m$ if all $m\times m$ minors of $A$ have determinants equal to $0$. Therefore the set $\text{Sing}(n,m):= M(n,m) \setminus M_+(n,m)$ is a submanifold of $M(n,m)$ of codimension at least $2$ (recall that $m < n$). Since the dimension of $M(n,m)$ is $nm$ and $\text{Sing}(n,m)$ has codimension at least $2$ we have 
\begin{equation}
	\label{eq-Hausdorff-measure}
	\mathcal{H}^{nm-1}(\text{Sing}(m,n))=0,
\end{equation}
where $\mathcal{H}^{nm-1}$ is the $(nm-1)$-dimensional Hausdorff measure on $M(n,m)$ relative to the Euclidean distance $d_\text{Eucl}$.
This property allows us to prove the following key property of the geometry of $M(n,m)$. This is the crucial ingredient in the proof of Theorem \ref{theo-intro-dist-equality} since it allows to construct good enough paths in $\Omega^1_+(M,\Rn)$ joining any two elements of $\Omega^1_+(M,\Rn)$.
\begin{lem}[Key Lemma.]
	\label{lemma-combination}
	Let $A_1,A_2,\ldots, A_k, B_2,\ldots, B_k \in M(n,m)$ such that $\sum_{i=1}^k A_i \in M_+(n,m)$. Then the set
	$$\left\lbrace B\in M(n,m) \text{ s.t. } \lin(A_1,B)(t) + \sum_{i=2}^k \lin(A_i,B_i)(t)\in M_+(n,m) \,\,\,\, \forall t\in [0,1]\right\rbrace$$
	is open and dense in $M(n,m)$, where  $\lin(A,B)(t) := (1-t)A + tB$ for $A, B \in M(n,m)$.
\end{lem}
Observe that the condition $\sum_{i=1}^k A_i \in M_+(n,m)$ is necessary to ensure that $\lin(A_1, B)(0) + \sum_{i=2}^k \lin(A_i,B_i)(0)\in M_+(n,m)$.
\begin{proof}
	It is enough to prove the statement for $k=2$. Indeed if $k\geq 3$ we set $A'=A_2+\cdots+A_k$ and $B'=B_2+\cdots +B_k$. Then the combination $\lin(A_1, B)(t) + \sum_{i=2}^k \lin(A_i,B_i)(t)$ equals $\lin(A_1, B)(t) + \lin(A',B')(t)$ and $A_1+A' = \sum_{i=1}^kA_i \in M_+(n,m)$.
	
	The openness follows from the continuity of the map $(B,t)\mapsto \lin(A_1, B)(t) + \lin(A_2,B_2)(t)$, the compactness of $[0,1]$ where $t$ belongs to, and the fact that $M_+(n,m)$ is open.	We can now pass to the density. Let $B_1\in M(n,m)$ and let $U$ be the Euclidean ball of radius $1$ centered at $B_1$ of the affine space $V$ at $B_1$ such that $V+B_2$ is orthogonal to the direction $\lin(A_1+ A_2, B_1+B_2)$. We define the map $Q\colon U \times [0,1] \to M(n,m)$ as $Q(B,t) := \lin(A_1 + A_2,B + B_2)(t)$. This is a bi-Lipschitz embedding. By \eqref{eq-Hausdorff-measure} we deduce that $\mathcal{H}^{nm-1}(Q^{-1} \text{Sing}(n,m)) = 0$. This implies that for almost every $B' \in U$ the whole segment $\lin(A_1 + A_2,B' + B_2)$ lies outside $\text{Sing}(n,m)$, since $U$ has dimension $nm-1$.
\end{proof}

A curve $c\colon [0,1] \to M_+(n,m)$ is said to be piecewise linear if it is the concatenation of $k$ linear segments $\lin(A_0, A_1)$, $\lin(A_1, A_2), \ldots$, $\lin(A_{k-1}, A_k)$ and $c$ is parametrized proportionally to arc-length with respect to the metric $g_{n\times m}$ \eqref{eq-metric-standard} in such a way that $\left\Vert \frac{d}{dt}c(t)\right\Vert_{g_{n\times m}} = L_{g_{n\times m}}(c)$. Here $L_{g_{n\times m}}$ denotes the length of a curve with respect to $g_{n\times m}$. Because of this parametrization we always have:
\begin{equation}
	\label{eq-estimate}
	\int_0^1 \left\Vert \frac{d}{dt}c(t)\right\Vert_{g_{n\times m}}^2 dt = L_{g_{n\times m}}(c)^2.
\end{equation}

The following lemma follows from the bi-Lipschitz equivalence between the Euclidean distance $d_{\text{Eucl}}$ and the distance $\distnm$ on each compact subset of $M_+(n,m)$.
\begin{lem}
	\label{lemma-comparison-length}
	For all compact subsets $K$ of $M_+(n,m)$ and all $\varepsilon > 0$ there exists $\delta > 0$ such that if $\lin(A,B)$ is a linear segment contained in $K$ of Euclidean length at most $\delta$ then $L_{g_{n\times m}}(\lin(A,B))<\varepsilon$.
\end{lem}

Also the following result is easy.

\begin{lem}
	\label{lemma-epsilon-delta}
	Let $c = \lin(A_0,\ldots,A_k)$ be a piecewise linear curve of $M_+(n,m)$. For all $\varepsilon > 0$ there exists $\delta > 0$ such that if $c' = \lin(A_0', \ldots, A_k')$ is a piecewise linear curve of $M_+(n,m)$ such that $d_{\textup{Eucl}}(A_i,A_i') < \delta$ for every $i=0,\ldots,k$ then $\vert L_{g_{n\times m}}(c') - L_{g_{n\times m}}(c)\vert < \varepsilon$.
\end{lem}
\begin{proof}
	Suppose it is not the case: we can find piecewise linear curves $c_h = \lin(A_{0,h}, \ldots, A_{k,h})$ such that $A_{j,h}$ converges to $A_j$ for every $j=0,\ldots,k$ but $L_{g_{n\times m}}(c_h)$ does not converge to $L_{g_{n\times m}}(c)$. However $\frac{d}{dt} c_h (t)$ converges to $\frac{d}{dt} c(t)$ in the tangent bundle $TM_+(n,m)$ for almost every $t \in [0,1]$, so by Lebesgue's Convergence Theorem we get $\lim_{h\to +\infty}L_{g_{n\times m}}(c_h) = L_{g_{n\times m}}(c)$, a contradiction.
\end{proof}
Finally we show that $\distnm$ can be computed using piecewise linear curves.
\begin{lem}
	\label{lemma-piecewise-approximation}
	Let $c\colon [0,1]\to M_+(n,m)$ be a piecewise $C^1$-curve and $\varepsilon > 0$. Then there exists a piecewise linear curve $c_\varepsilon$ with same endpoints of $c$ and such that $\vert L_{g_{n\times m}}(c_\varepsilon) - L_{g_{n\times m}}(c) \vert < \varepsilon$.
\end{lem}
\begin{proof}
	Let $0=t_0<t_1<\ldots < t_k = 1$ be a partition of $[0,1]$. If the partition is thin enough then the linear segments $\lin(c(t_i), c(t_{i+1}))$ belong to $M_+(n,m)$ for every $i=0,\ldots,k-1$ because $M_+(n,m)$ is open and $c([0,1])$ is compact. This allow us to define the piecewise linear curve $c_k = \lin(c(t_0), c(t_1), \ldots, c(t_k))$ in $M_+(n,m)$ associated to each thin enough partition. The sequence $\frac{d}{dt} c_k(t)$ converges to $\frac{d}{dt} c(t)$ in $TM_+(n,m)$ for almost every $t\in [0,1]$ (namely all $t$'s that do not belong to $\bigcup_k \lbrace t_k \rbrace$ and to the non-differentiable points of $c$), so $\lim_{k\to +\infty}L_{g_{n\times m}}(c_k) =  L_{g_{n\times m}}(c)$.
\end{proof}

\subsection{The proof of Theorem \ref{theo-intro-dist-equality}}\label{subsec.dist.inequality}
In this section, we prove the main result of the paper. One of the two inequalities is easy and it was already proved in \cite[Theorem~4.5]{bauer2021oneforms}. Indeed it is there shown that
\begin{align}\label{eq.distO.inequality}
	\distO(\alpha,\beta)^2\geq\int_Md_x(\alpha(x),\beta(x))^2d\mu_{g_0}(x)
\end{align}
for $\alpha, \beta\in\Omega^1_+(M,\Rn)$. 
\begin{proof}[Proof of Theorem \ref{theo-intro-dist-equality}]
	We just need to prove the reverse inequality in \eqref{eq.distO.inequality}.	Let $\alpha,\beta \in \Omega^1_+(M,\mathbb{R}^n)$. By openness of $M_+(n,m)$ in $M(n,m)$ and by continuity of $\alpha$ and $\beta$, for all $x\in M$ and all $\varepsilon > 0$ we can find a neighbourhood $U_x^\varepsilon$ of $x$ supporting a standard trivialization $\varphi_x^\varepsilon$ such that 
	\begin{itemize}
		\item[(i)] the segments $\lin\left(\varphi_x^\varepsilon\vert_x(\alpha(x)), \varphi_x^\varepsilon\vert_y(\alpha(y))\right)$ and $\lin\left(\varphi_x^\varepsilon\vert_x(\beta(x)), \varphi_x^\varepsilon\vert_y(\beta(y))\right)$ are contained in $M_+(n,m)$ for all $y\in U_x^\varepsilon$;
		\item[(ii)] the lengths of the segments above is smaller than $\varepsilon$ for all $y\in U_x^\varepsilon$, namely
		$$ L_{g_{n\times m}}\left(\lin(\varphi_x^\varepsilon\vert_x(\alpha(x)), \varphi_x^\varepsilon\vert_y(\alpha(y)))\right) < \varepsilon$$
		and the same for $\beta$.
	\end{itemize}
	For every $x\in M$ we can apply Lemma \ref{lemma-piecewise-approximation} to find a piecewise linear curve $c_x^\varepsilon \subseteq M_+(n,m)$ with endpoints $\varphi_x^\varepsilon\vert_x(\alpha(x))$ and $\varphi_x^\varepsilon\vert_x(\beta(x))$ and such that 
	$$L_{g_{n\times m}}(c_x^\varepsilon)<\distnm(\varphi_x^\varepsilon\vert_x(\alpha(x)),\varphi_x^\varepsilon\vert_x(\beta(x))) + \varepsilon.$$
	By compactness we extract a finite covering $\lbrace U_i^\varepsilon \rbrace$ from the covering $\lbrace U_x^\varepsilon \rbrace$. Let us call $\varphi_i^\varepsilon$ the trivializing chart for $U_i^\varepsilon$ and $c_i^\varepsilon = \lin(A_{0,i},\ldots,A_{k(i),i})$ the piecewise linear curve associated to this neighbourhood. By finiteness we can suppose without loss of generality that $k(i)=k$ for each $i$, maybe adding some constant subpath. We define $\Gamma_i^\varepsilon \colon [0,1] \to \Omega^1_+(U_i^\varepsilon,\mathbb{R}^n)$ by
	\begin{equation*}
		\begin{aligned}
			\Gamma_i^\varepsilon(\cdot)(y) &= \lin\left(\alpha(y), \varphi_i^{\varepsilon}\vert_y^{-1}(A_{0,i}),\ldots,\varphi_i^{\varepsilon}\vert_y^{-1}(A_{k,i}),\beta(y)\right)\\
			&= \varphi_i^{\varepsilon}\vert_y^{-1}\left(\lin\left(\varphi_i^{\varepsilon}\vert_y(\alpha(y)), A_{0,i}, \ldots, A_{k,i}, \varphi_i^{\varepsilon}\vert_y(\beta(y))\right)\right)
		\end{aligned}
	\end{equation*}
	Condition (i) guarantees that this curve is effectively full-ranked. Moreover $\Gamma_i^\varepsilon$ is piecewise $C^1$ and satisfies $\Gamma_i^\varepsilon(0) = \alpha\vert_{U_i^\varepsilon}$, $\Gamma_i^\varepsilon(1) = \beta\vert_{U_i^\varepsilon}$ and $$L_{g_y}(\Gamma_i^\varepsilon(\cdot)(y)) < d_y(\alpha(y),\beta(y)) + 3\varepsilon$$ 
	for all $y \in U_i^\varepsilon$ by (ii).
	
	Let $\lbrace \rho_i^\varepsilon \rbrace$ be a partition of unity associated to the covering $\lbrace U_i^\varepsilon \rbrace$. 
	We define the map $\Gamma^\varepsilon \colon [0,1] \to \Omega^1(M,\mathbb{R}^n)$ by $t\mapsto (y\mapsto \Sigma_i\rho_i^\varepsilon(y)\Gamma_i^\varepsilon(t)(y))$. It is piecewise $C^1$ but it could happen that $\Gamma^\varepsilon(t)(y)$ is not full-ranked for some $(t,y)\in [0,1]\times M$.
	Let $\delta > 0$ be the number given by Lemma \ref{lemma-epsilon-delta} relative to $\varepsilon$ and all the paths $\lin(A_{0,i}, \ldots,A_{k,i}) \subseteq M_+(n,m)$. \\
	\textbf{Claim:} we can find points $\tilde{A}_{0,i}, \ldots, \tilde{A}_{k,i} \in M_+(n,m)$ with $d_\text{Eucl}(A_{j,i}, \tilde{A}_{j,i}) < \delta$ in such a way that if we define the path 
	$$\tilde{\Gamma}_i^\varepsilon(\cdot)(y) := \lin\left(\alpha(y), \varphi_i^{\varepsilon}\vert_y^{-1}(\tilde{A}_{0,i}),\ldots,\varphi_i^{\varepsilon}\vert_y^{-1}(\tilde{A}_{k,i}),\beta(y)\right)$$
	then the section $\tilde{\Gamma}^\varepsilon(t)(y) := \Sigma_i\rho_i(y)\tilde{\Gamma}_i^\varepsilon(t)(y) \in \Omega^1(M,\RR^n)$ is full ranked for all $t\in [0,1]$ and $y\in M$. 
	
	For $y\in M$ we define $I_y$ to be the set of indices $i$ such that $\rho_i^\varepsilon(y)>0$, in particular $y\in U_i^\varepsilon$. We define also $p_y = \max I_y$ and for $p\in \NN$ we set $M_p = \lbrace y \in M \text{ s.t. } p_y = p\rbrace$. Clearly $M = \bigcup_{p \in \NN}M_p$. We define $\tilde{A}_{0,p}$ by induction on $p$. For $p=1$ we set $\tilde{A}_{0,1} = A_{0,1}$. For $p\geq 2$ we proceed as follows. Let $y\in M_p$. For all $i\in I_y$ we set $B_{0} = \varphi_p^\varepsilon\vert_y(\alpha(y))$ and $C_{0,i} = \varphi_p^\varepsilon\vert_y\varphi_i^\varepsilon\vert_y^{-1}(\tilde{A}_{0,i})$.
	We consider the set $\mathcal{A}_{0,p}(y)$ defined as
	$$\left\lbrace C \in M(n,m) \text{ s.t. } \rho_p^\varepsilon(y)\lin(B_{0},C) + \sum_{i\in I_y \setminus \lbrace p \rbrace}\rho_i^\varepsilon(y)\lin(B_{0},C_{0,i}) \subseteq M_+(n,m) \right\rbrace.$$
	Since 
	$$\sum_{i\in I_y} \rho_i^\varepsilon(y)B_{0} = \varphi_p^\varepsilon\vert_y\left(\sum_{i\in I_y}\rho_i^\varepsilon(y)\alpha(y)\right) = \varphi_p^\varepsilon\vert_y(\alpha(y)) \in M_+(n,m)$$
	we can apply Lemma \ref{lemma-combination} to conclude that $\mathcal{A}_{0,p}(y)$ is open and dense in $M(n,m)$. Moreover by openness of $M_+(n,m)$ and the continuity of all functions involved we can find a neighbourhood $V_y$ of $y$ in $M_p$ such that $\mathcal{A}_{0,p}(y') \supseteq \mathcal{A}_{0,p}(y)$ for all $y'\in V_y$. We cover $M_p$ by countable open sets $V_{y_\ell}$ as above. For all $y\in M_p$ we have $\mathcal{A}_{0,p}(y) \supseteq \bigcap_\ell \mathcal{A}_{0,p}(y_\ell)$, the latter being open and dense in $M(n,m)$ because $M(n,m)$ is a Baire space. We take as $\tilde{A}_{0,p}$ any point of $M(n,m)$ such that $d_\text{Eucl}(\tilde{A}_{0,p}, A_{0,p}) < \delta$ and $\tilde{A}_{0,p} \in \bigcap_\ell \mathcal{A}_{0,p}(y_\ell)$.
	
	Let us show that the first segment of the path $\tilde{\Gamma}$ defined using $\tilde{A}_{0,i}$ in place of $A_{0,i}$ is full ranked. Let $t\in [0,\frac{1}{k+2}]$ and $y\in M$. In order to show that $\tilde{\Gamma}(t)(y)$ is full ranked it is enough to check that $\varphi_p^\varepsilon\vert_y(\tilde{\Gamma}(t)(y)) \in M_+(n,m)$ for $p=p_y$. We have
	\begin{equation}
		\label{eq-induction}
		\begin{aligned}
			\varphi_p^\varepsilon\vert_y(\tilde{\Gamma}(t)(y)) &= \varphi_p^\varepsilon\vert_y\left(\sum_{i\in I_y} \rho_i^\varepsilon(y) \cdot \lin\left(\alpha(y), \varphi_i^\varepsilon\vert_y^{-1}(\tilde{A}_{0,i})\right)(t)\right) \\
			&= \sum_{i\in I_y} \lin\left(\rho_i^\varepsilon(y)B_0, \rho_i^\varepsilon(y)C_{0,i}\right)(t) \\
			&= \rho_p^\varepsilon(y)\lin\left(B_{0},\tilde{A}_{0,p}\right)(t) + \sum_{i\in I_y \setminus \lbrace p \rbrace}\rho_i^\varepsilon(y)\lin\left(B_{0},C_{0,i}\right)(t)
		\end{aligned}
	\end{equation}
	and the latter belongs to $M_+(n,m)$ by construction.
	
	We now proceed to define $\tilde{A}_{1,p}$ by induction on $p$. The construction is the same, the only difference is that we define $B_{0,i} = \varphi_p^\varepsilon\vert_y\varphi_i^\varepsilon\vert_y^{-1}(\tilde{A}_{0,i})$ instead of $B_0$. So we can define the set $\mathcal{A}_{1,p}$ similarly to $\mathcal{A}_{0,p}$. We have to check that we can apply Lemma \ref{lemma-combination} to get that $\mathcal{A}_{1,p}$ is open and dense. But by \eqref{eq-induction} we get
	\begin{equation*}
		\begin{aligned}
			\sum_{i\in I_y} \rho_i^\varepsilon(y)B_{0,i} &= \varphi_p^\varepsilon\vert_y\left(\sum_{i\in I_y}\rho_i^\varepsilon(y)\varphi_i^\varepsilon\vert_y^{-1}(\tilde{A}_{0,i})\right)\\
			&= \varphi_p^\varepsilon\vert_y\left(\tilde{\Gamma}\left(\frac{1}{k+2}\right)(y)\right) \in M_+(n,m).
		\end{aligned}
	\end{equation*}
	Proceeding by induction on $j$ we can prove the claim, paying attention to the case $j=k$ in order to ensure that also the last segments from $\varphi_i^\varepsilon\vert_y^{-1}\tilde{A}_{k,i}$ to $\beta(y)$ is full-ranked. But this can be done by intersecting the open and dense set $\mathcal{A}_{k,p}$ to the open and dense set defined as $\mathcal{A}_{0,p}$ with $\beta$ in place of $\alpha$.
	
	The path $\tilde{\Gamma}^\varepsilon$ we found above has the following properties: it is full ranked, i.e. it has values in $\Omega^1_+(M,\mathbb{R}^n)$, $\tilde{\Gamma}^\varepsilon(0) = \alpha$ and $\tilde{\Gamma}^\varepsilon(1) = \beta$. Moreover 
	$$L_{g_y}(\tilde{\Gamma}_i^\varepsilon(\cdot)(y)) \leq L_{g_y}(\Gamma_i^\varepsilon(\cdot)(y)) + 3\varepsilon < d_y(\alpha(y),\beta(y)) + 4\varepsilon$$ 
	for all $y\in M$ such that $i\in I_y$ because of the choice of $\delta$, Lemma \ref{lemma-epsilon-delta} and the fact that the map $\varphi_i^\varepsilon\vert_y$ is an isometry between $g_y$ and $g_{n\times m}$.
	
	We can now estimate the distance between $\alpha$ and $\beta$ using this curve:
	\begin{align}
		\distO(\alpha, \beta)^2 &\leq \left(\int_0^{1} \left\Vert \frac{d}{dt}\tilde{\Gamma}^\varepsilon(t)(\cdot)\right\Vert_{G} dt\right)^2 \\
		&\leq \int_0^{1}\left\Vert \frac{d}{dt}\tilde{\Gamma}^\varepsilon(t,\cdot)\right\Vert_{G}^2 dt\\
		&= \int_0^{1} \left\Vert \Sigma_{i}\rho_i^\varepsilon(\cdot)\frac{d}{dt}\tilde{\Gamma}_i^\varepsilon(t)(\cdot)\right\Vert_{G}^2 dt \\
		&=\int_0^{1}\int_M \left\Vert \Sigma_{i}\rho_i^\varepsilon(y)\frac{d}{dt}\tilde{\Gamma}_i^\varepsilon(t)(y)\right\Vert_{g_y}^2d\mu_{g_0}(y)\\
		&\leq \int_0^{1} \int_M (\Sigma_{i\in I_y}\rho_i^\varepsilon(y))^2 \cdot\max_{i\in I_y}\left\Vert\frac{d}{dt}\tilde{\Gamma}_i^\varepsilon(t)(y)\right\Vert_{g_y}^2 d\mu_{g_0}(y) dt \\
		&= \int_M \int_0^{1} \max_{i\in I_y}\left\Vert\frac{d}{dt}\tilde{\Gamma}_i^\varepsilon(t)(y)\right\Vert_{g_y}^2 dt d\mu_{g_0}(y),
	\end{align}
	where the first inequality follows by Cauchy-Schwarz inequality.
	For all $i\in I_y$ we have $$\int_0^{1}\left\Vert\frac{d}{dt}\tilde{\Gamma}_i^\varepsilon(t)(y)\right\Vert_{g_y}^2 \leq L_{g_y}(\tilde{\Gamma}_i^\varepsilon(t)(y))^2 \leq (d_y(\alpha(y),\beta(y))+ 4\varepsilon)^2$$
	because of \eqref{eq-estimate}. Therefore
	$$\distO(\alpha,\beta) \leq \int_M (d_y(\alpha(y),\beta(y))+ 4\varepsilon)^2 d\mu_{g_0}(y).$$
	The thesis follows by taking $\varepsilon$ going to $0$.
\end{proof}

\begin{rem}
	The same proof, without all the approximation details, works word by word to the case of the $L^2$-metric on the space of Riemannian metrics, providing a new, short proof of Theorem 3.8 of \cite{clarke2013geodesics}. Since this computation is of independent interest, it is explicitly done in Appendix~\ref{appendix:Ebin}.
\end{rem}

\subsection{The completion of the space of full-ranked one-forms}
In the finite dimensional case, any Riemannian tensor always induces a metric space structure on the manifold. However, this does not always hold in the infinite dimensional case. Examples of such Riemannian tensors with induced vanishing geodesic distances can be found in \cite{eliashberg1993bi,michor2005vanishing,bauer2013geodesic,bauer2018vanishing}, where the induced geodesic distances fail to be positive definite. Moreover, it is also not clear a priori that the induced geodesic distance between two points is actually finite. Theorem \ref{theo-intro-dist-equality} says that $\distO$ is a true distance.
\begin{cor}
	The function $\distO$ on $\Omega^1_+(M,\Rn)$ is a finite distance.
\end{cor}
\begin{proof}
	The positive definiteness of $\distO$ was already observed in \cite{bauer2021oneforms} as a consequence of \eqref{eq.distO.inequality}, while the finiteness follows from Theorem \ref{theo-intro-dist-equality} because of the continuity of the map $y\mapsto d_y(\alpha(y), \beta(y))$ for $\alpha,\beta \in \Omega^1_+(M,\Rn)$ and the compactness of $M$.
\end{proof}
In \cite[Theorem 4.9]{bauer2021oneforms} it is shown that the metric space $(\Omega^1_+(M,\Rn), \distO)$ is always incomplete by computing explicitly the geodesics of this metric space. A more refined expression for these geodesics can be found in Appendix~\ref{appendix:geo}.

Our next goal is to give a description of the completion of the space of full-ranked one-forms $(\Omega^1_+(M,\Rn), \distO)$, analogue to the one given by the first author in \cite{cavallucci2022} for the space of Riemannian metrics. 
\begin{thm}
	\label{theo-completion}
	Let $M$ be a compact, orientable, $m$-dimensional smooth manifold and let $n>m$. Then the metric completion of $(\Omega^1_+(M,\Rn), \distO)$ is isometric to $(L^2(M,\overline{M_+(n,m)}), d_{L^2})$. In particular it depends only on the dimension of $M$ and on $n$.
\end{thm}
The space $L^2(M,\overline{M_+(n,m)})$ appearing in the statement of the theorem is the space of $L^2$-maps from $M$ to the metric completion $\overline{M_+(n,m)}$ of $(M_+(n,m), \distnm)$, i.e. the set of maps $f\colon M \to \overline{M_+(n,m)}$ such that $\int_M \distnm(f(x), A)^2d\mu_{g_0}(x) < +\infty$ for some (hence any) $A\in \overline{M_+(n,m)}$ modulo the equivalence relation $f\sim g$ if and only if $f(x)=g(x)$ for $\mu_{g_0}$-a.e. $x\in M$.
Moreover the distance $d_{L^2}$ on $L^2(M,\overline{M_+(n,m)})$ is the well defined expression:
$$d_{L^2}(f,g)^2 = \int_M \distnm(f(x), g(x))^2d\mu_{g_0}(x).$$
\begin{proof}
	The proof of this theorem follows verbatim the scheme of \cite{cavallucci2022}. We just sketch the main steps and highlight a technical detail. As in \cite{cavallucci2022} we denote by $\mathcal{E}_{C^\infty}$ (resp. $\mathcal{E}_{C^0}$ , $\mathcal{E}_{L^2}$) the set of smooth (resp. continuous, $L^2$) sections of the fiber bundle $(\Rn\otimes T^*M)_+$. 
	We equip all these spaces with the $L^2$-metric
	$$d_{L^2}(\alpha, \beta)^2 :=
	\int_M d_y(\alpha(y), \beta(y))^2d\mu_{g_0}(y),$$
	which, by Theorem \ref{theo-intro-dist-equality}, coincides with $\distO$ on $\mathcal{E}_{C^\infty} = \Omega^1_+(M,\Rn)$. The next step is to show that $\mathcal{E}_{C^\infty}$ is dense in $\mathcal{E}_{L^2}$ with respect to the $L^2$-metric. This can be done in two steps: first we show that $\mathcal{E}_{C^\infty}$ is dense in $\mathcal{E}_{C^0}$ and then that $\mathcal{E}_{C^0}$ is dense in $\mathcal{E}_{L^2}$. In \cite[Theorem 2]{cavallucci2022} these facts are proved using the Tietze's Extension Theorem in the version of \cite[Theorem 4.1]{Dug51} and classical approximation procedures with partitions of unity. In our case we can use the same proof to approximate a $L^2$-section by a continuous one, but which is not a priori full-ranked. The same problem appears when we use partitions of unity. This technical difficulty can be overcome arguing as in the proof of Theorem \ref{theo-intro-dist-equality}. Indeed we can further approximate these possibly non-full-ranked continuous sections by full-ranked ones. The details are omitted. We showed that $\Omega^1_+(M,\Rn) = \mathcal{E}_{C^\infty}$ is dense in $\mathcal{E}_{L^2}$ with the $L^2$-metric.
	
	At this point we can continue as in \cite{cavallucci2022}. We can now define another space. The fiber at $x\in M$ of the fiber bundle $(\Rn \otimes T^*M)_+$ is equipped with the metric $d_x$, which is isometric to $d_{n\times m}$. We can take a new fiber bundle whose fiber at $x$ is the completion of $d_x$, which is isometric to $\overline{M_+(n,m)}$. The space of $L^2$-sections of this new fiber bundle, equipped with the $L^2$-distance, is denoted by $\overline{\mathcal{E}}_{L^2}$. We have a natural inclusion $\mathcal{E}_{L^2} \subseteq \overline{\mathcal{E}}_{L^2}$, and we can show as in \cite[Theorem 2]{cavallucci2022} that $\mathcal{E}_{L^2}$ is actually dense.
	
	The last step is to show as in \cite[Theorem 3]{cavallucci2022} that $\overline{\mathcal{E}}_{L^2}$ with the $L^2$-metric is isometric to the space $L^2(M,\overline{M_+(n,m)})$. In particular it is complete, so it is the completion of $\Omega^1_+(M,\Rn)$.

	Finally the space $(L^2(M,\overline{M_+(n,m)}), d_{L^2})$ is isometric to $L^2([0,1],\overline{M_+(n,m)})$ with the $L^2$-distance, as follows by the proof of \cite[Theorem 1]{cavallucci2022}, so the completion of $\Omega^1_+(M,\Rn)$ depends only on $m$ and $n$.
\end{proof}

\subsection{An alternative description of the completion}
We now give an alternative description of the metric completion for the space of full-ranked one-forms $\Omega^1_+(M,\Rn)$ solving Conjecture 4.10 of \cite{bauer2021oneforms}, which states that the metric completion of $\Omega^1_+(M,\Rn)$ is given as the space of finite volume sections of the vector bundle $\Rn\otimes T^*M$ modulo the equivalence relation $\alpha\sim \beta$ if and only if
\begin{align}
	\alpha(x)\neq \beta(x) \Longrightarrow \operatorname{rank}(\alpha(x)) < m \text{ and } \operatorname{rank}(\beta(x)) < m
\end{align}
holds $\mu_{g_0}$-almost everywhere. Here we say that a section $\alpha: M\to \Rn\otimes T^*M$ has finite volume if $\alpha$ is measurable and $$\int_M\sqrt{\det(g_0(x)^{-1}(\alpha^T\alpha)(x))}d\mu_{g_0}(x)<\infty.$$
Denote by $\hat{\Omega}^1_+(M,\Rn)$ the above formulation of the metric completion of the space of full-ranked one-forms. 
It is proved in \cite[Theorem 3.13]{bauer2021oneforms} that the metric completion $\overline{M_+(n,m)}$ can be identified with $M_+(n,m) \cup \lbrace 0 \rbrace$ where $\distnm(A,0)$ is the infimum of the $g_{n\times m}$-lengths of all curves $c\colon [0,1]\to M(n,m)$ starting at $A$ such that $c\vert_{[0,1)}\subseteq M_+(n,m)$ and $c(1)\in \text{Sing}(n,m)$. 
The following proposition tells us that the metric completion formulation $\hat{\Omega}^1_+(M,\Rn)$ of the space of full-ranked one-forms in \cite[Conjecture~4.10]{bauer2021oneforms} can be identified with the $L^2$ space of mappings from $M$ to $\overline{M_+(n,m)}$ we give in Theorem~\ref{theo-completion}.


\begin{prop}
	\label{prop-bijection}
	There is an explicit bijection between the space $\hat{\Omega}^1_+(M,\Rn)$ and the space $L^2(M,\overline{M_+(n,m)})$.
\end{prop}
\begin{proof}
	We first define a map from the space $\hat{\Omega}^1_+(M,\Rn)$ to $L^2(M,\overline{M_+(n,m)})$. Fix a triangulation of $M$ and call $V_i$ the interior of the simplices: they are disjoint and $\mu_{g_0}(M\setminus \bigcup_i V_i) = 0$. We can also suppose that each $V_i$ is the support of a standard trivialization $\varphi_i$. We define a map
	\begin{align}
		F: \hat{\Omega}^1_+(M,\Rn) &\to L^2(M,\overline{M_+(n,m)})\\
		[\alpha]&\mapsto [\overline{\alpha}]
	\end{align}
	by
	\begin{align}
		\overline{\alpha}(x) = \begin{cases}
			(\pi_2\circ\varphi_i\circ\alpha)(x) &\quad x\in V_i \text{ and it is full ranked,}\\
			0 &\quad \text{ otherwise.}
		\end{cases}
	\end{align}
	The map $\overline{\alpha}$ is defined $\mu_{g_0}$ almost everywhere, which is enough to define a measurable map from $M$ to $\overline{M_+(n,m)}$. It is also a $L^2$-map because of the finite volume assumption on $\alpha$, as it follows by \cite[Lemma 3.12]{bauer2021oneforms}. If $[\alpha] = [\beta] \in \hat{\Omega}^1_+(M,\Rn)$ then $\overline{\alpha}(x)=\overline{\beta}(x)$ for $\mu_{g_0}$-a.e.$(x)$ by definition, so the map $F$ is well-defined.
	
	If $d_{L^2}([\overline{\alpha}], [\overline{\beta}]) = 0$ then by definition we have 
	\begin{align}
		\distnm(\overline{\alpha}(x), \overline{\beta}(x)) = 0
	\end{align}
	for $\mu_{g_0}$-a.e. $x\in M$. Therefore for $\mu_{g_0}$-a.e. $x\in M$ it holds that either $\alpha(x) = \beta(x)$ or they are both not full ranked, i.e. $[\alpha] = [\beta]$ in $\hat{\Omega}^1_+(M,\Rn)$, so $F$ is injective.
	
	Now let $[\overline{\alpha}]\in L^2(M,\overline{M_+(n,m)})$. Then locally on each $V_i$, $\overline{\alpha}$ can be seen as a $L^2$-section of the trivial bundle $V_i\times \overline{M_+(n,m)}$, and the following gives a preimage $[\alpha]$ in $\hat{\Omega}^1_+(M,\Rn)$
	\begin{align}
		\alpha(x) = \begin{cases}
			(\varphi_i^{-1}\circ \overline{\alpha}\vert_{V_i})(x) &\quad \text{ for }  x\in V_i \text{ and } \alpha(x)\neq 0,\\
			0 \in (\Rn\otimes T^*M ) &\quad \text{ otherwise}.
		\end{cases}
	\end{align}
	So defined $\alpha$ is a measurable section. Using the next Lemma \ref{lem.matrix.bounds} and $[\overline{\alpha}]$ is $L^2$, we have
	\begin{align}
		\int_M\sqrt{\det(g_0)(x)^{-1}(\alpha^T\alpha))(x)}d\mu_{g_0}(x) &= \int_M\sqrt{\det(\overline{\alpha}^T\overline{\alpha})(x)}dx\\
		&\leq \frac{m}{4}\int_M\distnm\left(\overline{\alpha}(x),0\right)dx \leq +\infty,
	\end{align}
	which shows that $\alpha$ has finite volume and the statement is proved.
\end{proof}

\begin{lem}\label{lem.matrix.bounds}
	Let $A, B\in M_+(n,m)$. Then
	\begin{align}
		\distnm(A, B) &\geq \frac{2}{\sqrt{m}}\left\vert\sqrt[4]{\det(A^TA)}-\sqrt[4]{\det(B^TB)}\right\vert .
	\end{align}
\end{lem}
\begin{proof}
	For any arbitrary $X\in M_+(n,m)$ and $W\in T_XM_+(n,m)$, the following inequality always holds
	\begin{equation}
		\begin{aligned}
			\label{eq-trace}
			\tr(W(X^TX)^{-1}W^T) &= \tr(\boldsymbol{WW}^T) = \tr(\boldsymbol{W_0W_0}^T) + \frac{(\tr(\boldsymbol{W}))^2}{m}\\
			&\geq \frac{(\tr(\boldsymbol{W}))^2}{m} = \frac{(\tr(WX^+))^2}{m},
		\end{aligned}
	\end{equation}
	where $\boldsymbol{W} = WX^+$ and $\boldsymbol{W_0} = \boldsymbol{W} - \frac{\tr(\boldsymbol{W})}{m}XX^+$ is the traceless part of $\boldsymbol{W}$. 
	Let $\varphi:[0,1]\to M_+(n,m)$ be a piecewise $C^1$-path connecting $A$ and $B$. Then we have
	\begin{align}
		\frac{d}{dt}\left(\sqrt[4]{\det(\varphi^T\varphi)}\right) &= \frac14\left(\det(\varphi^T\varphi)\right)^{-3/4}\tr\left(\frac{d}{dt}(\varphi^T\varphi) \cdot (\varphi^T\varphi)^{-1}\right)\det(\varphi^T\varphi)\\
		&=\frac14\tr\left(\frac{d}{dt}(\varphi^T\varphi)\cdot(\varphi^T\varphi)^{-1}\right)\sqrt[4]{\det(\varphi^T\varphi)}\\
		&=\frac12\tr(\varphi_t\varphi^+)\sqrt[4]{\det(\varphi^T\varphi)}\\
		&=\frac12\left(\left(\tr(\varphi_t\varphi^+)\right)^2\sqrt{\det(\varphi^T\varphi)}\right)^{1/2}\\
		&\leq\frac{\sqrt{m}}{2}\left(\tr(\varphi_t(\varphi^T\varphi)^{-1}\varphi_t^T)\sqrt{\det(\varphi^T\varphi)}\right)^{1/2},
	\end{align}
	where the last inequality follows by \eqref{eq-trace}. Therefore
	\begin{align}
		\sqrt[4]{\det(A^TA)}-\sqrt[4]{\det(B^TB)} &= \int_0^1\frac{d}{dt}\left(\sqrt[4]{\det(\varphi^T\varphi)}\right)dt\\
		&\leq\frac{\sqrt{m}}{2}\int_0^1\left(\tr(\varphi_t(\varphi^T\varphi)^{-1}\varphi_t^T)\sqrt{\det(\varphi^T\varphi)}\right)^{1/2}dt\\
		&=\frac{\sqrt{m}}{2}L_{g_{n\times m}}(\varphi).
	\end{align}
	The result follows immediately by exchanging $A$ and $B$ and taking the infimum of the length over $\varphi$. 
\end{proof}

\section{Connection to the space of metrics}
\label{sec.EbinMetric}

Let $\Gamma(S^2T^*M)$ denote the space of all smooth symmetric $(0,2)$ tensor fields on $M$. The space of all Riemannian metrics on $M$, denoted by $\Met(M)$, is the space of all smooth positive definite symmetric $(0,2)$ tensor fields on $M$, and thus an open subset of the space $\Gamma(S^2T^*M)$. Its tangent space at each element $g$ is the space $\Gamma(S^2T^*M)$ itself. Let $g\in\Met(M)$ and $h,k\in T_g\Met(M)\cong\Gamma(S^2T^*M)$. We consider the following metric defined as a multiple of the Ebin metric \cite{ebin1970manifold} on $\Met(M)$:
\begin{align}\label{eq.metric.RieMetrics}
	( h, k)_g = \frac14\int_M\tr\left(g^{-1}hg^{-1}k\right)d\mu_g,
\end{align}
where $d\mu_g$ is the volume form induced by $g$.

In local coordinates with respect to the basis $\{dx^i, i=1,2,\cdots, m\}$ of $T^*_xM$, each symmetric $(0,2)$ tensor field on $M$ restricted to $x\in M$ can be represented by a $m\times m$ symmetric matrix, and is positive definite if the tensor field is a Riemannian metric. Denote by $\Sym(m)$ the space of symmetric $m\times m$ matrices and by $\Sym_+(m)$ the open subspace of positive definite symmetric $m\times m$ matrices. The space $\Sym_+(m)$ is a $\frac{m(m+1)}{2}$
dimensional manifold with tangent space at each point given by $\Sym(m)$. 
Let $g\in \Sym_+(m)$ and $h, k\in T_g\Sym_+(m) \cong \Sym(m)$. The Riemannian metric on $\Sym_+(m)$ corresponding to \eqref{eq.metric.RieMetrics} is:
\begin{align}\label{eq.metric.sym}
	\langle h, k\rangle^{\Sym}_g = \frac14\left(g^{-1}hg^{-1}k\right)\sqrt{\det(g)}.
\end{align}
We denote by $\distMet$ the distance function on $\Met(M)$ with respect to the Riemannian metric \eqref{eq.metric.RieMetrics}, and by $\distSym$ the distance function on $\Sym_+(m)$ with respect to the Riemannian metric \eqref{eq.metric.sym}. By \cite{cavallucci2022} we know that the metric completion of $(\Met(M), \distMet)$ is isometric to $(L^2(M,\overline{\Sym_+(m)}), d_{L^2})$, where $\overline{\Sym_+(m)}$ is the metric completion of the space of positive definite symmetric matrices $(\Sym_+(m), \distSym)$.

Note that each $\alpha\in\Omega^1_+(M,\Rn)$ induces a Riemannian metric $\alpha^T\alpha$ on $M$. We can then define a mapping from the space of full-ranked one-forms to the space of Riemannian metrics on $M$ as follows
\begin{align}\label{eq.proj.oneForm}
	\tilde{\pi}: \Omega^1_+(M,\Rn)\to\Met(M),\quad \alpha\mapsto\alpha^T\alpha.
\end{align}
Restrict to one point $x\in M$ we have the corresponding mapping from the space of full rank matrices to the space of positive definite symmetric matrices
\begin{align}\label{eq.proj.matrix}
	\pi \colon M_+(n,m) \to \Sym_+(m), \quad A \mapsto A^TA.
\end{align}
These two mappings \eqref{eq.proj.oneForm} and \eqref{eq.proj.matrix} have proven to be Riemannian submersions \cite[Theorem~4.2, 3.3]{bauer2021oneforms}, in particular $1$-Lipschitz.

\subsection{The quotient structure of \texorpdfstring{$\Met(M)$}{Lg}}
For the finite dimensional case the fibers of the submersion \eqref{eq.proj.matrix} has been characterized as the orbits under the action of the orthogonal group $\on{O}(n)$ \cite[Proposition~3.2]{bauer2021oneforms}. Note that the infinite dimensional group of smooth mappings $C^{\infty}(M, \SO(n))$ acts naturally on the space of full-ranked one-forms $\Omega^1_+(M,\Rn)$ by pointwise left multiplication. We now show the following proposition for the infinite dimensional case, which tells us that the orbits under the action of $C^{\infty}(M, \SO(n))$ are precisely the fibers of the mapping \eqref{eq.proj.oneForm} and thus gives a quotient structure of the space of Riemannian metrics $\Met(M)$. 
\begin{prop}\label{prop.pointwiseOrbits}
	Let $\alpha, \beta\in\Omega_+^1(M, \Rn)$. Then $\tilde{\pi}(\alpha) = \tilde{\pi}(\beta)$ if and only if there exists a function $O\in C^{\infty}(M,\on{SO}(n))$ such that $\alpha =O\beta$.
\end{prop}
\begin{proof}
	It is straightforward to see that $\alpha = O\beta$ with $O\in C^{\infty}(M,\SO(n))$ implies that $\alpha^T\alpha = \beta^T\beta$. Note that the Moore-Penrose inverse $\alpha^+(x)$ is the adjoint operator of $\alpha(x)$ from $\RR^n$ to $T_xM$, see Remark~\ref{rem.penroseinv}. For the opposite direction, we will show that there is a function $O\in C^{\infty}(M, \SO(n))$ such that $\alpha\alpha^+ = O\beta\alpha^+$, then the final result follows directly from $\alpha^+\alpha = \on{Id}$, where $\on{Id}: M\times\RR^n\to M\times\RR^n$ denotes the identity bundle map over $M$. 
	
	Both $\alpha\alpha^+$ and $\beta\alpha^+$, for each $x\in M$, are linear transformations from $\RR^n$ to itself of rank $m$. In other words, $\alpha\alpha^+, \beta\alpha^+: M\times \RR^n\to M\times \RR^n$ are both bundle homomorphisms of constant rank $m$ from the trivial vector bundle $M\times \RR^n$ to itself. Therefore, the image of the bundle map $\alpha\alpha^+$, denoted by $\on{Im}(\alpha\alpha^+)$, is a smooth rank-$m$ subbundle $M\times\RR^m$ \cite[Theorem 10.34]{lee2003introduction}. Note that $\alpha^+\alpha$ is the identity bundle map from the tangent bundle $TM$ to itself and it has eigenvalue $1$ of multiplicity $m$, hence the linear bundle map $\alpha\alpha^+$ has eigenvalue $1$ of multiplicity $m$ and eigenvalue $0$ of multiplicity $n-m$. Note that $
	(\on{Id} - \alpha\alpha^+)(\alpha\alpha^+) = \alpha\alpha^+ - \alpha\alpha^+ = 0.
	$
	Thus, the image $\on{Im}(\alpha\alpha^+)=M\times\RR^m$ is just the $m$-dimensional eigenspace of the linear operator $\alpha\alpha^+$ associated with eigenvalue $1$. 
	
	We now choose a smooth orthonormal basis $\{u_1,\cdots, u_m\}$ of sections of the trivial bundle $\on{Im}(\alpha\alpha^+)=M\times\RR^m$ and extend it to 
	a smooth global frame $\{u_1,\cdots,u_m, u_{m+1},\cdots,u_n\}$ for $M\times\RR^n$. We obtain
	\begin{align}
		(\alpha\alpha^+)(u_1,\cdots, u_n) = (u_1,\cdots,u_n)\Sigma,\quad \Sigma = \begin{pmatrix}
			I_m & 0\\
			0 & 0
		\end{pmatrix}
	\end{align}
	Define $\{v_i, i=1,\cdots,m \}$ such that $v_i = (\beta\alpha^+)(u_i)$ for $i = 1,\cdots,m$. Since $\alpha^T\alpha = \beta^T\beta$, we have  $\left(\beta\alpha^+\right)^T = \alpha\beta^+ $. Then it follows from
	\begin{align}
		\langle v_i, v_j\rangle_{\RR^n} &= \left\langle (\beta\alpha^+)(u_i), (\beta\alpha^+)(u_i)\right\rangle_{\RR^n}\\
		&= \left\langle u_i, (\alpha\alpha^+)(u_j)\right\rangle_{\RR^n}\\
		&= \langle u_i, u_j\rangle_{\RR^n},
	\end{align}
	that $\{v_1,\cdots,v_m\}$ is a set of smooth orthonormal sections of $M\times\RR^n$. Extend $\{v_i\}$ to a smooth global frame $\{v_1,\cdots,v_n\}$ for $M\times\RR^n$. Immediately we have for each $x\in M$ that
	\begin{align}
		(\beta\alpha^+)(u_1,\cdots, u_n) = (v_1,\cdots,v_n)\Sigma.
	\end{align}
	Let $U$ and $V$, for each $x\in M$, be the orthonormal matrices with columns formed by $\{u_1,\cdots,u_n\}\vert_x$ and by $\{v_1,\cdots,v_n\}\vert_x$, respectively. Note that $m< n$. We can always change the sign of the last columns $u_n$ and $v_n$ such that $U(x)$ and $V(x)$ are special orthogonal matrices. Since the columns are all global smooth sections of $M\times\Rn$, we have $U, V\in C^{\infty}(M, \SO(n))$. Let $O = UV^{-1}$. The result follows immediately.
\end{proof}
Note that the group action of $C^\infty(M,\SO(n))$ on $(\Omega_+^1(M,\RR^n), \distO)$ is by isometries \cite[Lemma 4.1]{bauer2021oneforms}. Proposition \ref{prop.pointwiseOrbits} directly implies that the metric space $(\Met(M), \distMet)$ is isometric to the quotient of $(\Omega_+^1(M,\RR^n), \distO)$ by the group of isometries $C^\infty(M,\SO(n))$.

\subsection{The quotient structure of the metric completion of \texorpdfstring{$\Met(M)$}{Lg}} The projection $\tilde{\pi}$ \eqref{eq.proj.oneForm} can be extended to the metric completions of the space of full-ranked one-forms $\Omega^1_+(M,\Rn)$ and the space of Riemannian metrics $\Met(M)$ with the same formula:
\begin{align}\label{eq.proj.metricCompletion}
	\tilde{\pi}:L^2(M,\overline{M_+(n,m)}) \to L^2(M,\overline{\Sym_+(m)}),\qquad \alpha\mapsto\alpha^T\alpha.
\end{align}
We want to describe naturally the space $L^2(M,\overline{\Sym_+(m)})$ also as a quotient of the space $L^2(M,\overline{M_+(n,m)})$ by a group of isometries. Let $\text{Mes}(M, \SO(n))$ denote the space of all measurable mappings from $M$ to the special orthogonal group $\SO(n)$. The operation $(O\cdot O') (x) := O(x)O'(x)$ induces on $\text{Mes}(M, \SO(n))$ a group structure which is continuous with respect to any topology which implies pointwise convergence. The group $\text{Mes}(M, \SO(n))$ acts in a natural way by isometries on the metric space $L^2(M,\overline{M_+(n,m)})$ as $(O \cdot \alpha)(x) = O(x)\alpha(x)$.

We show the following proposition analogous to Proposition~\ref{prop.pointwiseOrbits}, and it immediately leads to a quotient characterization of the metric completion of the space of Riemannian metrics $\Met(M)$. 
\begin{prop}
	Let $\tilde{\pi}$ be as in \eqref{eq.proj.metricCompletion}. Then $\tilde{\pi}(\alpha) = \tilde{\pi}(\beta)$ if and only if $\alpha = O\beta$ for some $O\in \textup{Mes}(M,\SO(n))$. In other words $L^2(M,\overline{\Sym_+(m)})$ is isometric to the quotient of $L^2(M,\overline{M_+(n,m)})$ by the group of isometries $\textup{Mes}(M,\SO(n))$. The map $\tilde{\pi}$ is a metric submetry.
\end{prop}
Recall that a submetry is a map $f\colon X \to Y$ between two metric spaces such that $f(\overline{B}(x,r)) = \overline{B}(f(x),r)$ for all $x\in X$ and $r\geq 0$, where $\overline{B}$ denotes a closed ball. The definition implies that every submetry is in particular surjective and $1$-Lipschitz. If $G$ is a group of isometries of a metric space and with closed orbits then the quotient map is a submetry, see \cite{KL22}.
\begin{proof}
	By definition if $\alpha = O\beta$ for some $O\in \text{Mes}(M,\SO(n))$ then $\tilde{\pi}(\alpha) = \tilde{\pi}(\beta)$.
	We now show the opposite direction. Note that $\tilde{\pi}(\alpha)=\alpha^T\alpha$ is positive semidefinite. For every $x$ on which $(\alpha^T\alpha)(x)$ is positive definite there exists a unique positive definite matrix, denoted by $p(x)$, such that $p^2(x) = (\alpha^T\alpha)(x)$. Let 
	\begin{align}
		\tilde{p}(x) = \begin{cases}
			\begin{pmatrix}
				p\\
				0
			\end{pmatrix} & \text{for } x \text{ with } (\alpha^T\alpha)(x) \text{ full rank}\\[.5em]
			0 & \text{otherwise}.
		\end{cases}
	\end{align}
	The map $\tilde{p}\colon M \to M_+(n,m)$ is measurable, and moreover
	\begin{align}
		\int_M\sqrt{\det(g_0(x)^{-1}(\tilde{p}^T\tilde{p})(x))}d\mu_{g_0}(x) &= \int_M\sqrt{\det(g_0(x)^{-1} p^2(x))}d\mu_{g_0}(x)\\[.1em] 
		&= \int_M\sqrt{\det(g_0(x)^{-1}(\alpha^T\alpha)(x))}d\mu_{g_0}(x)<\infty,
	\end{align}
	so it naturally defines an element of $L^2(M,\overline{M_+(n,m)})$ via the bijection of Proposition \ref{prop-bijection}. Since we are supposing $\tilde{\pi}(\alpha)=\alpha^T\alpha = \beta^T\beta = \tilde{\pi}(\beta)$ we conclude that the map $\tilde{p}$ above is the same, $\mu_{g_0}$-a.e., for $\alpha$ and $\beta$. So it is enough to show that there exists $O\in \text{Mes}(M, \SO(n))$ such that $\alpha = O\tilde{p}$. Define
	\begin{align}
		O_1 = \begin{cases}
			\alpha p^{-1} & \text{for } x \text{ with } \alpha(x) \text{ full rank}\\[.5em]
			\begin{pmatrix}
				I_{m\times m} \\
				0
			\end{pmatrix} & \text{otherwise}
		\end{cases}
	\end{align}
	with $I_{m\times m}$ being the $m\times m$ identity matrix. Then we have
	\begin{align}
		(O_1^TO_1)(x) = I_{m\times m},
	\end{align}
	which means that the columns of $O_1(x)$ gives a set of orthonormal vectors with respect to the Euclidean scalar product. Let $O_2(x)$ be the matrix whose columns are an orthonormal basis of the orthogonal complement of the span of columns of $O_1(x)$ and let $O(x) = \begin{pmatrix}
		O_1(x) & O_2(x)
	\end{pmatrix}$. Observe that the choice of $O_2(x)$ can be made in a measurable way using the Gram-Schmidt procedure since $O_1$ is measurable. In addition, we can always change the sign of the last column in $O_2(x)$ such that $O\in \text{Mes}(M,\SO(n))$. Thus
	\begin{align}
		\alpha = O\tilde{p}.
	\end{align}
	Indeed $\alpha(x) = O(x)\tilde{p}(x)$ is clear for the points $x$ such that $\alpha(x)$ has full rank. For the points $x$ where $\alpha(x)$ has no full rank then also $O(x)\tilde{p}(x)$ has no full rank, so they are identified in $\overline{M_+(n,m)}$.
	
	In order to show the last part of the statement we need to prove that the orbits under $\text{Mes}(M,\SO(n))$ are closed. Let $\alpha_j = O_j\alpha_0$ be elements in the orbit of $\alpha_0 \in L^2(M,\overline{M_+(n,m)})$ and suppose $\alpha_j \to \alpha_\infty$ with respect to the distance $d_{L^2}$. For every $x\in M$ the elements $O_j(x)$ belong to the compact space $\SO(n)$, so up to a subsequence they converge to $O_\infty(x)$. Observe that the limit does not depend on the particular subsequence for $\mu_{g_0}$-almost every $x$, because it must hold $\alpha_\infty(x) = O_\infty(x)\alpha_0(x)$. Thus $O_\infty(x)$ is well defined for $\mu_{g_0}$-almost every $x$. Note that it is the limit of measurable functions so $O_\infty \in \text{Mes}(M,\SO(n))$ and clearly $\alpha_\infty = O_\infty \alpha_0$, which shows that the orbits are closed.
\end{proof}

\appendix
\section{Geodesics in the space of full-ranked one-forms}
\label{appendix:geo}
	We present an explicit formula for geodesics in the space of full-ranked one-forms $\Omega^1_+(M, \Rn)$, which refines the one given in \cite{bauer2021oneforms}, offering detailed expressions for the variables within the formula.
	\begin{prop}\label{thm.geo}
		Let $\alpha\in \Omega^1_+(M,\Rn)$ and $\zeta\in T_{\alpha}\Omega^1_+(M,\Rn)$. Then the geodesic in $\Omega^1_+(M,\Rn)$ starting at $\alpha$ in the direction of $\zeta$ is the curve
		\begin{align}\label{eq.geo}
			\alpha(t) = f(t)^{\frac1m}e^{s(t)(Z_0 - Z_0^T)}e^{s(t)Z_0^T\alpha\alpha^+}\alpha,
		\end{align}
		where $Z_0 = Z - \frac{\tr(Z)}{m}\alpha\alpha^+$ is the traceless part of $Z = \zeta\alpha^+$, and where
		\begin{align}
			f(t) &= \frac{m}{4}\tr(Z_0Z_0^T)t^2 + \left(1+\frac{\tr(Z)}{2} t\right)^2\\
			s(t) &=\begin{cases}
				\frac{2}{\sqrt{m\tr(Z_0Z_0^T)}}\arctan\left(\frac{\sqrt{m\tr(Z_0Z_0^T)}t}{2+\tr(Z)t}\right) & Z_0 \neq 0\\[.5em]
				\frac{t}{1 + \frac{t}{2}\tr(Z)} & Z_0 = 0.
			\end{cases}
		\end{align}
		In particular, the change in the induced volume element is given by
		\begin{align}
			\sqrt{\det(\alpha(t)^T\alpha(t))} = f(t)\sqrt{\det(\alpha^T\alpha)}.
		\end{align}
	\end{prop}
	\begin{proof}
		This theorem is basically a reformulation of \cite[Theorem~3.6]{bauer2021oneforms} of the geodesic equation on the space of full rank matrices $M_+(n,m)$. The pointwise nature of the metric \eqref{eq.metric.OneForms} allows one to translate the result directly to the space of full-ranked one-forms $\Omega^1_+(M,\Rn)$. Using $\alpha^+\alpha = I_{m\times m}$ and $Z\alpha\alpha^+ = Z$, with the same $f(t), s(t), \omega_0$ and $P_0$ as in \cite[Theorem~3.6]{bauer2021oneforms} we compute
		\begin{align}
			\delta &= \tr(Z^TZ) = \tr\left(\left(Z_0 + \frac{tr(Z)}{m}\alpha\alpha^+\right) \left(Z_0^T + \frac{tr(Z)}{m}\alpha\alpha^+\right)\right)\\
			& = \tr(Z_0Z_0^T) + \frac{2\tr(Z)}{m}\tr(Z_0\alpha\alpha^+) + \frac{(\tr(Z))^2}{m}\\
			& = \tr(Z_0Z_0^T) + \frac{(\tr(Z))^2}{m}.
		\end{align}
		Then 
		\begin{align}\label{eq.f}
			f(t) &=  \frac{m\delta}{4}t^2+\tau t+1\\
			& = \frac{m}{4}\left(\tr(Z_0Z_0^T) + \frac{(\tr(Z))^2}{m}\right)t^2 + \tr(Z) t + 1\\
			&= \frac{m}{4}\tr(Z_0Z_0^T)t^2 + \frac{(\tr(Z))^2}{4} t^2 + \tr(Z) t + 1\\
			&= \frac{m}{4}\tr(Z_0Z_0^T)t^2 + \left(1+\frac{\tr(Z)}{2} t\right)^2,
		\end{align}
		and thus
		\begin{align}\label{eq.s}
			s(t) &= \int_0^t\frac{d\sigma}{f(\sigma)} = \int_0^t\frac{1}{\frac{m}{4}\tr(Z_0Z_0^T)\sigma^2 + \left(1+\frac{\tr(Z)}{2} \sigma\right)^2}d\sigma\\
			&= 4\int_0^t \frac{d\sigma}{m\tr(Z_0Z_0^T)\sigma^2 + (2+\tr(Z)\sigma)^2}\\
			&=\begin{cases}
				\frac{2}{\sqrt{m\tr(Z_0Z_0^T)}}\arctan\left(\frac{\sqrt{m\tr(Z_0Z_0^T)}t}{2+\tr(Z)t}\right) &  Z_0 \neq 0\\[.5em]
				\frac{t}{1 + \frac{t}{2}\tr(Z)} & Z_0 = 0.
			\end{cases}
		\end{align}
		Note that for any arbitrary $m\times m$ matrix $A$ and full rank $n\times m$ matrix $a$,
		\begin{align}
			e^{aAa^+} = I_{n\times n} - \alpha\alpha^+ + ae^Aa^+.
		\end{align}	
		It follows that $\alpha e^{s(t)P} = e^{s(t)\alpha P\alpha^+}\alpha$. By computation 
		\begin{align}
			\alpha P\alpha^+ &= \alpha\left((\alpha^T\alpha)^{-1}(\zeta^T\alpha) - \frac{\tau}{m}I_{m\times m}\right)\alpha^+\\
			&= Z^T\alpha\alpha^+ - \frac{\tr(Z)}{m}\alpha\alpha^+\\
			&= Z_0^T\alpha\alpha^+.
		\end{align}
		It is easy to see that $\omega = Z - Z^T = Z_0 - Z_0^T$. We then have a reformulation of the geodesic formula
		\begin{align}
			\alpha(t) &= f(t)^{1/m}e^{-s(t)\omega}\alpha e^{s(t)P}\\
			&= f(t)^{1/m}e^{s(t)(Z_0 - Z_0^T)}e^{s(t)Z_0^T\alpha\alpha^+}\alpha,
		\end{align}
		where $f(t)$ and $s(t)$ are as in \eqref{eq.f} and \eqref{eq.s}. 
		In addition, let $\zeta_0 =  \zeta - \frac{\tr(Z)}{m}\alpha$ and we compute 
		\begin{align}
			\alpha(t)^T\alpha(t) &= f^{2/m}\alpha^Te^{s(t)\alpha\alpha^TZ_0}e^{s(t)Z_0^T\alpha\alpha^+}\alpha\\
			&= f^{2/m}e^{s(t)\alpha^T\zeta_0(\alpha^T\alpha)^{-1}}(\alpha^T\alpha)e^{s(t)(\alpha^T\alpha)^{-1}\zeta_0^T\alpha\alpha}.
		\end{align}
		Using the property of the matrix exponential that $\det(e^{A}) = e^{\tr(A)}$ we obtain
		\begin{align}
			\det(\alpha(t)^T\alpha(t)) &= f(t)^2e^{\tr(s(t)\alpha^T\zeta_0(\alpha^T\alpha)^{-1})}\det(\alpha^T\alpha)e^{\tr(s(t)(\alpha^T\alpha)^{-1}\zeta_0^T\alpha)}\\
			&=f(t)^2e^{s(t)\tr(Z_0)}\det(\alpha^T\alpha)e^{s(t)\tr(Z_0)}\\
			&=f(t)^2\det(\alpha^T\alpha).
		\end{align}
		The second statement follows immediately.
	\end{proof}
	It is easy to see from the geodesic formula \eqref{eq.geo} that the geodesic is only defined for $t\in[0, t_0]$ with $t_0 = -\frac{2}{-\tr(Z)}$ for the points of $M$ where $Z_0 = 0$ and $\tr(Z)<0$. As a direct consequence, we have geodesically incompleteness and the metric incompleteness of the space of full-ranked one-forms $\Omega^1_+(M,\Rn)$ and the space of full rank matrices $M_+(n,m)$.
	
	\section{Theorem \ref{theo-intro-dist-equality} for the Ebin metric}\label{appendix:Ebin}
	The aim of this appendix is to show an alternative proof of the equivalent of Theorem \ref{theo-intro-dist-equality} for the Ebin metric on the space of Riemannian metrics of a compact, connected, orientable manifold without boundary, which simplifies the argument of \cite[Theorem 3.8]{clarke2013geodesics}. In the sequel $M$ is a compact, connected, orientable, $m$-dimensional manifold without boundary. The space Met$(M)$ is by definition the set of smooth sections of the fiber bundle $E=S^2_+T^*M$ of positive definite symmetric $(0,2)$-tensors. It is equipped with the Riemannian structure $(\cdot,\cdot)_{L^2}$ defined by \eqref{eq.metric.RieMetrics}. It can be expressed equivalently with the help of a fixed Riemannian metric $g_0$ of volume $1$ of $M$, as observed for instance in \cite{clarke2013geodesics}. Indeed to each fiber $E_x = S^2_+T_x^*M$ of the fiber bundle $E$ is associated the Riemannian metric 
	\begin{equation}
		\langle a,b \rangle_{h,x} = \frac14\tr(h^{-1}ah^{-1}b)\sqrt{\det(g_0(x)^{-1}h)},
	\end{equation}
	where $h \in S^2_+T_x^*M$ and $a,b \in T_h S^2_+T_x^*M \cong S^2T_x^*M$. For this appendix we denote this metric by $g_x$ and its induced distance by $d_x$. The metric \eqref{eq.metric.RieMetrics} can be equivalently defined as
	\begin{equation}
		\label{eq.metric.Riem.def}
		(h,k)_g = \frac{1}{4}\int_M \langle h(x),k(x)\rangle_{g(x),x} d\mu_{g_0}(x).
	\end{equation}
	Recall that the induced distance on $\Met(M)$ is denoted by $\distMet$. 
	\begin{thm}
		\label{theo-appendix}
		In the situation above it holds
		$$\distMet(g,g')^2 = \int_M d_x(g(x),g'(x))^2 d\mu_{g_0}(x)$$
		for all $g,g'\in \Met(M)$.
	\end{thm}
	
	We need a bit of preparation, in the same spirit of Section \ref{sec.L2metric}.\\
	A local trivialization $\varphi$ of the bundle $E$ induced by a local chart on an open set $U$ sending the volume form $\mu_{g_0}$ to the Euclidean one is called \emph{standard}. Let $\varphi\vert_x \colon E_x \to \Sym_+(m)$ be the restriction of $\varphi$ to the fiber at $x\in U$. Then the pushforward of the metric on $ E_x$, namely $(\varphi\vert_x)_*(\langle \cdot,\cdot \rangle_{\cdot,x})$, defines the Riemannian structure \eqref{eq.metric.sym} on $\Sym_+(m)$. We denote it by $g_{m,+}$. Recall that its induced distance is denoted by $\distSym$. Therefore if $\varphi$ is a standard local trivialization around $x\in M$ then $\varphi\vert_x$ is an isometry between
	$d_x$ and $\distSym$.
	
	A curve $c\colon [0,1] \to \Sym_+(m)$ is said to be piecewise linear if it is the concatenation of $k$ linear segments $\lin(A_0, A_1)$, $\lin(A_1, A_2), \ldots$, $\lin(A_{k-1}, A_k)$ and $c$ is parametrized proportionally to arc-length with respect to $g_{m,+}$ in such a way that $\left\Vert \frac{d}{dt}c(t)\right\Vert_{g_{m,+}} = L_{g_{m,+}}(c)$. Here the right hand side is the length of the curve with respect to the Riemannian metric $g_{m,+}$.
	Because of this parametrization we always have:
	\begin{equation}
		\label{eq-estimate—Ebin}
		\int_0^1 \left\Vert \frac{d}{dt}c(t)\right\Vert_{g_{m,+}}^2 dt = L_{g_{m,+}}(c)^2.
	\end{equation}
	The distance $\distSym$ can be computed using piecewise linear curves.
	\begin{lem}
		\label{lemma-piecewise-approximation-riem}
		Let $c\colon [0,1]\to \Sym_+(m)$ be a piecewise $C^1$-curve and $\varepsilon > 0$. Then there exists a piecewise linear curve $c_\varepsilon$ with same endpoints of $c$ and such that $\vert L_{g_{m,+}}(c_\varepsilon) - L_{g_{m,+}}(c) \vert < \varepsilon$.
	\end{lem}
	The proof is the same of Lemma \ref{lemma-piecewise-approximation}.
	
	\begin{proof}[Proof of Theorem \ref{theo-appendix}] As observed in \cite[Theorem 2.1]{clarke2013geodesics} the left hand side is always bigger than or equal the right hand side. So what we need to prove is
		$$\distMet(g, g')^2 \leq \int_M d_x(g(x),g'(x))^2d\mu_{g_0}(x).$$
		For all $x\in M$ and all $\varepsilon > 0$ we can find a neighbourhood $U_x^\varepsilon$ of $x$ supporting a standard trivialization $\varphi_x^\varepsilon$ such that 
		\begin{itemize}
			\item[(i)] the segments $\lin(\varphi_x^\varepsilon\vert_x(g(x)), \varphi_x^\varepsilon\vert_y(g(y)))$ and $\lin(\varphi_x^\varepsilon\vert_x(g'(x)), \varphi_x^\varepsilon\vert_y(g'(y)))$ are contained in $\Sym_+(m)$ for all $y\in U_x^\varepsilon$, by convexity of $\Sym_+(m)$;
			\item[(ii)] the lengths of the segments above is smaller than $\varepsilon$ for all $y\in U_x^\varepsilon$, namely
			$$ L_{g_{m,+}}(\lin(\varphi_x^\varepsilon\vert_x(g(x)), \varphi_x^\varepsilon\vert_y(g(y)))) < \varepsilon$$
			and the same for $g'$.
		\end{itemize}
		
		For every $x\in M$ we can apply Lemma \ref{lemma-piecewise-approximation-riem} to find a piecewise linear curve $c_x^\varepsilon \subseteq \Sym_+(m)$ with endpoints $\varphi_x^\varepsilon\vert_x(g(x))$ and $\varphi_x^\varepsilon\vert_x(g'(x))$ and such that 
		$$L_{g_{m,+}}(c_x^\varepsilon)< \distSym(\varphi_x^\varepsilon\vert_x(g(x)),\varphi_x^\varepsilon\vert_x(g'(x))) + \varepsilon.$$
		
		By compactness we extract a finite covering $\lbrace U_i^\varepsilon \rbrace$ from the covering $\lbrace U_x^\varepsilon \rbrace$. Let us call $\varphi_i^\varepsilon$ the trivializing chart for $U_i^\varepsilon$ and $c_i^\varepsilon = \lin(A_{0,i},\ldots,A_{k(i),i})$ the piecewise linear curve associated to this neighbourhood. By finiteness we can suppose without loss of generality that $k(i)=k$ for each $i$, maybe adding some constant subpath. We define $\Gamma_i^\varepsilon \colon [0,1] \to \textup{Met}(U_i^\varepsilon)$ by
		\begin{equation*}
			\begin{aligned}
				\Gamma_i^\varepsilon(\cdot)(y) &= \lin(g(y), \varphi_i^{\varepsilon}\vert_y^{-1}(A_{0,i}),\ldots,\varphi_i^{\varepsilon}\vert_y^{-1}(A_{k,i}),g'(y))\\
				&= \varphi_i^{\varepsilon}\vert_y^{-1}(\lin(\varphi_i^{\varepsilon}\vert_y(g(y)), A_{0,i}, \ldots, A_{k,i}, \varphi_i^{\varepsilon}\vert_y(g'(y)))).
			\end{aligned}
		\end{equation*}
		Observe that for each fixed $y\in U_i^\varepsilon$ this is a piecewise linear curve living on $S^2T_y^*M$.
		
		Condition (i) guarantees that this curve is positive-definite. Moreover $\Gamma_i^\varepsilon$ is piecewise $C^1$ and satisfies $\Gamma_i^\varepsilon(0) = g\vert_{U_i^\varepsilon}$, $\Gamma_i^\varepsilon(1) = g'\vert_{U_i^\varepsilon}$ and
		\begin{equation}
			\label{eq-Gamma-i}
			L_{g_y}(\Gamma_i^\varepsilon(\cdot)(y)) < d_y(g(y),g'(y)) + 3\varepsilon
		\end{equation} 
		for all $y \in U_i^\varepsilon$ by (ii).
		
		Let $\lbrace \rho_i^\varepsilon \rbrace$ be a partition of unity associated to the covering $\lbrace U_i^\varepsilon \rbrace$. For $y\in M$ we define $I_y$ to be the set of indices $i$ such that $\rho_i^\varepsilon(y)>0$, in particular $y\in U_i^\varepsilon$.\\	
		We define the map $\Gamma^\varepsilon \colon [0,1] \to \textup{Met}(M)$ by $t\mapsto (y\mapsto \Sigma_i\rho_i^\varepsilon(y)\Gamma_i^\varepsilon(t)(y))$. It is piecewise $C^1$ and $\Gamma^\varepsilon(t)(y)$ is positive definite for all $(t,y)\in [0,1]\times M$ because $\Sym_+(m)$ is convex. Moreover  $\Gamma^\varepsilon(0) = g$ and $\Gamma^\varepsilon(1) = g'$. 
		We can now estimate the distance between $g$ and $g'$ using this curve:
		\begin{align}
			d_{L^2}(g, g')^2 &\leq \left(\int_0^{1} \left\Vert \frac{d}{dt}\Gamma^\varepsilon(t)(\cdot)\right\Vert dt\right)^2 \\
			&\leq \int_0^{1}\left\Vert \frac{d}{dt}\Gamma^\varepsilon(t)(\cdot)\right\Vert^2 dt\\
			&= \int_0^{1} \left\Vert \Sigma_{i}\rho_i^\varepsilon(\cdot)\frac{d}{dt}\Gamma_i^\varepsilon(t)(\cdot)\right\Vert^2 dt \\
			&=\int_0^{1}\int_M \left\Vert \Sigma_{i}\rho_i^\varepsilon(y)\frac{d}{dt}\Gamma_i^\varepsilon(t)(y)\right\Vert_{g_y}^2d\mu_{g_0}(y)\\
			&\leq \int_0^{1} \int_M (\Sigma_{i\in I_y}\rho_i^\varepsilon(y))^2 \cdot\max_{i\in I_y}\left\Vert\frac{d}{dt}\Gamma_i^\varepsilon(t)(y)\right\Vert_{g_y}^2 d\mu_{g_0}(y) dt \\
			&= \int_M \int_0^{1} \max_{i\in I_y}\left\Vert\frac{d}{dt}\Gamma_i^\varepsilon(t)(y)\right\Vert_{g_y}^2 dt d\mu_{g_0}(y).
		\end{align}
		For all $i\in I_y$ we have $$\int_0^{1}\left\Vert\frac{d}{dt}\Gamma_i^\varepsilon(t)(y)\right\Vert_{g_y}^2 \leq L_{g_y}(\Gamma_i^\varepsilon(t)(y))^2 \leq (d_y(g(y),g'(y))+ 3\varepsilon)^2$$
		because of \eqref{eq-estimate—Ebin} and \eqref{eq-Gamma-i}. Therefore
		$$d_{L^2}(g,g') \leq \int_M (d_y(g(y),g'(y))+ 3\varepsilon)^2 d\mu_{g_0}(y).$$
		The thesis follows by taking $\varepsilon$ going to $0$.
	\end{proof}

\bibliographystyle{abbrv}
\bibliography{refs}
\end{document}